\pdfoutput=1
\documentclass[a4paper,reqno, 12pt,titlepage]{amsart}

\usepackage[utf8]{inputenc}
\usepackage{microtype}

\usepackage{amssymb, graphicx, enumerate, enumitem, color}
\usepackage[usenames,dvipsnames,svgnames,table]{xcolor}
\usepackage[foot]{amsaddr}

\colorlet{mylinkcolor}{violet}
\colorlet{mycitecolor}{YellowOrange}
\colorlet{myurlcolor}{Aquamarine}
\usepackage[unicode=true]{hyperref}
\hypersetup{
    pdftitle={Sparsity and dimension},
    pdfauthor={Gwena\"el Joret, Piotr Micek, Veit Wiechert},
    colorlinks,
    linkcolor={red!50!black},
    citecolor={blue!50!black},
    urlcolor={blue!80!black}
}

\newtheorem{theorem}{Theorem}
\newtheorem{conjecture}[theorem]{Conjecture}

\newtheorem{problem}[theorem]{Problem}

\newtheorem{proposition}[theorem]{Proposition}

\newtheorem{lemma}[theorem]{Lemma}
\theoremstyle{remark}

\newcommand{\set}[1]{\{#1\}}
\newcommand{\norm}[1]{{\left|#1\right|}}
\newcommand{\setR}{\mathbb{R}}

\newcommand{\calX}{\mathcal{X}}
\newcommand{\calC}{\mathcal{C}}

\newcommand{\calF}{\mathcal{F}}

\newcommand{\N}{\mathbb{N}}
\newcommand{\R}{\mathbb{R}}

\DeclareMathOperator\Inc{Inc}
\DeclareMathOperator\cover{cover}
\DeclareMathOperator\tw{tw}

\DeclareMathOperator\girth{girth}

\let\leq\leqslant
\let\geq\geqslant

\let\subset\subseteq
\let\subsetneq\varsubsetneq

\let\epsilon\varepsilon
\let\epsi\varepsilon

\renewenvironment{enumerate}{\begin{enumorig}[label=\textup{(\roman*)}, noitemsep, topsep=2pt plus 2pt, labelindent=.2em, leftmargin=*, widest=iii]}{\end{enumorig}}

\renewenvironment{itemize}{\begin{itemorig}[label=\textbullet, noitemsep, topsep=2pt plus 2pt, labelindent=.5em, labelsep=.5em, leftmargin=*]}{\end{itemorig}}

\makeatletter
\let\old@setaddresses\@setaddresses
\def\@setaddresses{\bigskip\bgroup\parindent 0pt\let\scshape\relax\old@setaddresses\egroup}
\makeatother

\begin{document}
\title[Sparsity and dimension]{Sparsity and dimension}

\author[G.~Joret]{Gwena\"el  Joret}
\address[G.~Joret]{Computer Science Department \\
  Universit\'e Libre de Bruxelles\\
  Brussels\\
  Belgium}
\email{gjoret@ulb.ac.be}

\author[P.~Micek]{Piotr Micek}
\address[P.~Micek]{Theoretical Computer Science Department\\
  Faculty of Mathematics and Computer Science, Jagiellonian University, Krak\'ow, Poland and Institut f\"ur Mathematik, Technische Universit\"at Berlin, Berlin, Germany}
\email{piotr.micek@tcs.uj.edu.pl}

\thanks{Piotr Micek was partially supported by the National Science Center of Poland under grant no.\ 2015/18/E/ST6/00299.}

\author[V.~Wiechert]{Veit Wiechert}
\address[V.~Wiechert]{Institut f\"ur Mathematik\\
  Technische Universit\"at Berlin\\
  Berlin \\
  Germany}
\email{wiechert@math.tu-berlin.de}

\thanks{G.\ Joret is supported by an ARC grant from the Wallonia-Brussels Federation of Belgium. 
V.\ Wiechert is supported by the Deutsche Forschungsgemeinschaft within the research training group `Methods for Discrete Structures' (GRK 1408).} 

\thanks{A preliminary version of this paper appeared as an extended abstract in the Proceedings of the Twenty-Seventh Annual ACM-SIAM Symposium on Discrete Algorithms (SODA '16)~\cite{SODAversion}.}

\date{\today}



\begin{abstract}
We prove that posets of bounded height whose cover graphs belong to a fixed class with bounded expansion have bounded dimension.
Bounded expansion, introduced by Ne\v{s}et\v{r}il and Ossona de Mendez as a model for sparsity in graphs, is a property that is naturally satisfied by a wide range of graph classes, from graph structure theory (graphs excluding a minor or a topological minor) to graph drawing (e.g.\ graphs with bounded book thickness).
Therefore, our theorem generalizes a number of results including the most recent one for posets of bounded height with cover graphs excluding a fixed graph as a topological minor.
We also show that the result is in a sense best possible, as it does not extend to nowhere dense classes; in fact, it already fails for cover graphs with locally bounded treewidth.
\end{abstract}
\maketitle

\section{Introduction}

\subsection{Poset dimension and cover graphs}
Partially ordered sets, \emph{posets} for short, are studied extensively in combinatorics, set theory and theoretical computer science.
One of the most important measures of complexity of a poset is its dimension.
The \emph{dimension} $\dim(P)$ of a poset $P$
is the least integer $d$ such that points of $P$ can be embedded into $\setR^d$ in such a way that $x< y$ in $P$ if and only if
the point of $x$ is below the point of $y$ with respect to the product order on $\setR^d$.
Equivalently, the dimension of $P$ is the least $d$ such that there are $d$ linear extensions of $P$ whose intersection is $P$.
Not surprisingly, dimension is hard to compute: Already deciding whether $\dim(P)\leq 3$ is an NP-hard problem~\cite{Yan82}, and for every $\epsi>0$, there is no $n^{1-\epsi}$-approximation algorithm for dimension unless ZPP=NP, where $n$ denotes the size of the input poset~\cite{CLN12}.
Research in this area typically focus on finding witnesses for large dimension and
sufficient conditions for small dimension; see e.g.~\cite{Tro-handbook} for a survey.

Posets are visualized by their Hasse \emph{diagrams}:
Points are placed in the plane and whenever $a<b$ in the poset, and there is no point $c$ with $a<c<b$,
there is a curve from $a$ to $b$  going upwards (that is $y$-monotone).
The diagram represents those relations which are essential in the sense that they are not implied
by transitivity, known as \emph{cover relations}.
The undirected graph implicitly defined by such a diagram is the \emph{cover graph} of the poset.
That graph can be thought of as encoding the `topology' of the poset.

There is a common belief that posets having a nice or well-structured drawing should have small dimension.
But first let us note a negative observation by Kelly~\cite{Kel81}, from 1981,
there is a family of posets whose diagrams can be drawn without edge crossings---undoubtedly qualifying as a `nice'  drawing---and with arbitrarily large dimension,  see Figure~\ref{fig:kelly}.
A key observation about Kelly's construction is that these posets also have large height.
This leads us to our main theorem, which generalizes several previous works in this area.

\begin{figure}[t]
 \centering
 \includegraphics[scale=1.0]{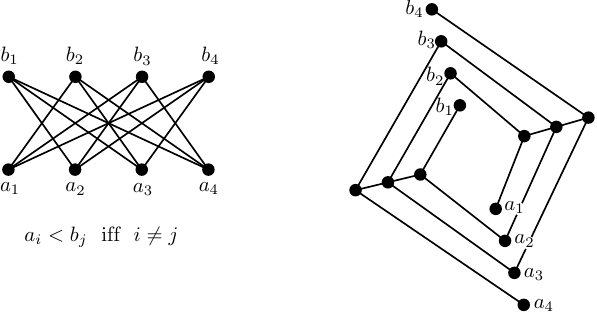}
 \caption{Standard example $S_4$ (left) and Kelly's construction of a planar poset containing $S_4$ (right).
 Recall that the {\em standard example $S_d$} is the poset on $2d$ points consisting of $d$ minimal points $a_1, \dots, a_d$ and $d$ maximal points $b_1, \dots, b_d$ such that $a_i < b_j$ in $S_d$ if and only if $i \neq j$. It is well-known and easily verified that $\dim(S_d)=d$. For every $d \geq 1$, Kelly's construction provides
 a planar poset with $4d-2$ points containing $S_d$ as a subposet, and hence having dimension at least $d$; its
general definition is implicit in the figure.}
 \label{fig:kelly}
\end{figure}

\begin{theorem}\label{thm:bounded-expansion}
For every  class of graphs $\calC$ with bounded expansion, and for every integer $h \geq 1$,
posets of height $h$ whose cover graphs are in $\calC$ have bounded dimension.
\end{theorem}

\subsection{Background}
Theorem~\ref{thm:bounded-expansion} takes its roots in the following result of Streib and Trotter~\cite{ST14}: Posets with planar cover graphs have dimension bounded in terms of their height.
Note that this justifies the fact that height and dimension grow together in Kelly's construction.
Joret, Micek, Milans, Trotter, Walczak, and Wang~\cite{JMMTWW} subsequently proved that the same result holds in the case of cover graphs of bounded treewidth, of bounded genus, and more generally cover graphs that forbid a fixed apex graph as a minor.
(Recall that a graph is \emph{apex} if it can be made planar by removing at most one vertex.)
In another direction, F{\"u}redi and Kahn~\cite{FK86} showed that posets with cover graphs of bounded maximum degree have dimension bounded in terms of their height.\footnote{We note that the original statement of F{\"u}redi and Kahn's theorem is that posets with \emph{comparability} graphs
of bounded maximum degree have bounded dimension; in fact, they show a $O(\Delta \log^{2} \Delta)$ upper bound on the dimension, where $\Delta$ denotes the maximum degree.
Observe however that the comparability graph of a poset has bounded maximum degree if and only its cover graph does and its height is bounded.}
All this was recently generalized by Walczak~\cite{Walczak17}:
Posets whose cover graphs exclude a fixed graph as a topological minor have dimension bounded by a function of their height.
(We note that Walczak's original proof relies on the graph structure theorems for graphs excluding a fixed topological minor; see~\cite{MW} for an elementary proof.)

Bounded degree graphs, planar graphs, bounded treewidth graphs, and more generally graphs avoiding a fixed (topological) minor are sparse, in the sense that they have linearly many edges.
This remark naturally leads one to ponder whether simply having a sparse cover graph is enough to guarantee the dimension be bounded by a function of the height.
This is exactly the question we address in this work.
Our main contribution is to characterize precisely the type of \emph{sparsity} that is needed to ensure the property.
Before explaining it let us make some preliminary observations.

\begin{figure}[!b]
 \centering
 \includegraphics[scale=0.74]{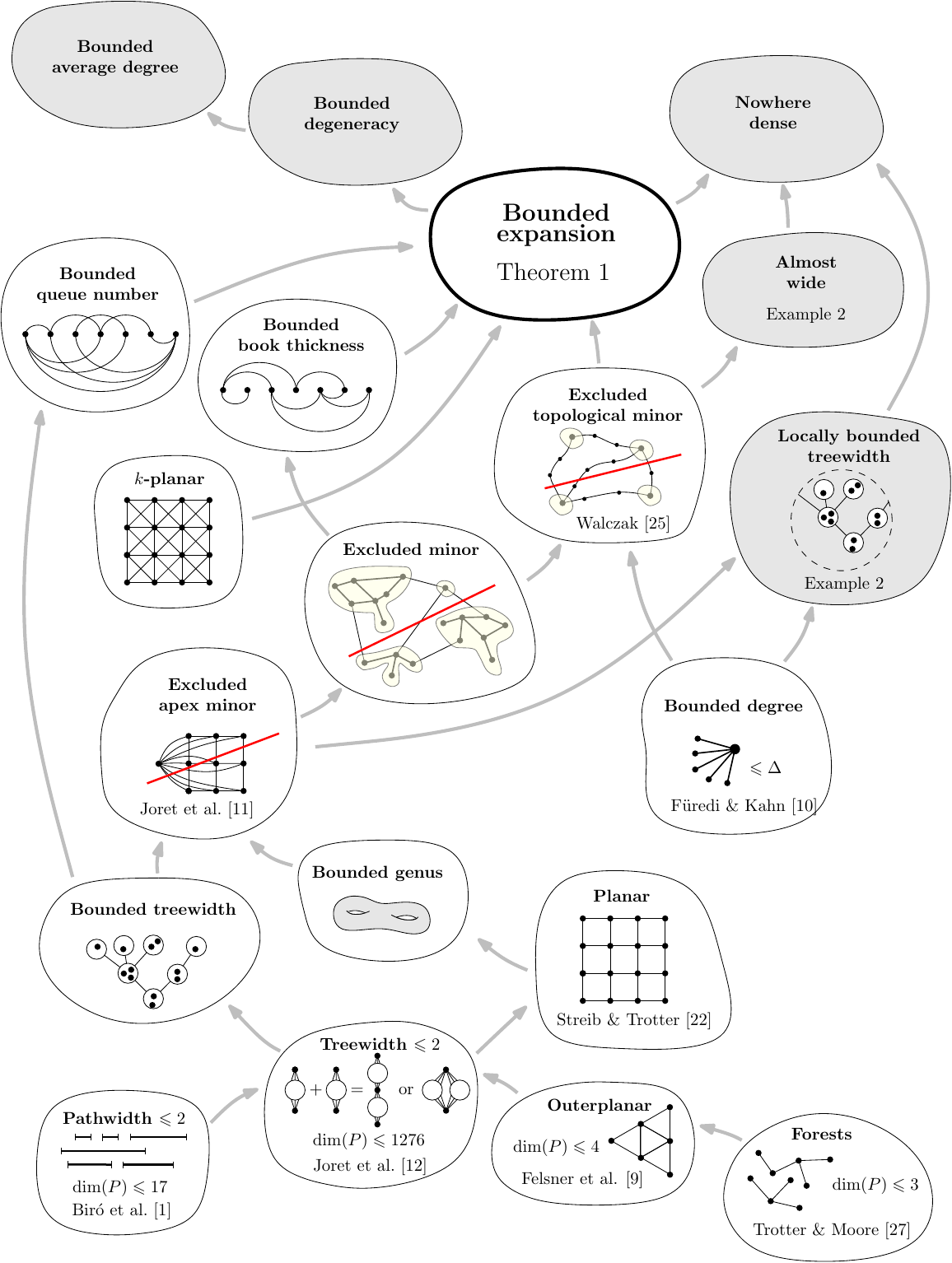}
 \caption{\label{fig:hierarchy}
 A hierarchy of sparse classes of graphs.
 A type of classes is drawn in grey whenever it includes a class $\calC$ such that there are posets of bounded height
 and unbounded dimension with cover graphs in $\calC$. 
 {\footnotesize (Remark: To improve readability some inclusions are not drawn, e.g.\ $k$-planar graphs have locally bounded treewith.)}
 }
\end{figure}

Clearly, asking literally that the cover graph has at most $c n$ edges for some constant $c$, where $n$ is the number of points of the poset, is not enough since one could simply consider the union of a standard example on $\Theta(\sqrt{n})$ points and a large antichain. 
Even requiring that property to hold for every subgraph of the cover graph (that is, bounding its degeneracy) is still not enough:
Define the \emph{incidence poset} $I_G$ of a graph $G$ as the height-$2$ poset with point set $V(G)\cup E(G)$,
where for a vertex $v$ and an edge $e$ we have $v < e$ in $P_G$ whenever $v$ is an endpoint of $e$ in $G$ (and no other relation).
The fact that $\dim(I_{K_n}) \geq \log\log n$ follows easily from repeated applications of the Erd\H{o}s-Szekeres theorem for monotone sequences (see for instance~\cite{DM41}).
Yet, the cover graph of $I_{K_n}$ is a clique where each edge is subdivided once, which is $2$-degenerate.

\subsection{Nowhere dense classes and classes with bounded expansion}
Ne\v{s}et\v{r}il and Ossona de Mendez~\cite{NOdM-book} carried out a thorough study of sparse classes of graphs over the last decade.
Most notably, they introduced the notions of \emph{nowhere dense} classes and classes with \emph{bounded expansion}.
Let us give some informal intuition, see Section~\ref{sec:prelim} for the precise definitions.
The key idea in both cases is to look at minors of \emph{bounded depth}, that is,
minors that can be obtained by first contracting disjoint connected subgraphs of bounded radius, and then possibly removing some vertices and edges.
In a nowhere dense class it is required that bounded-depth minors exclude at least some graph (which can depend on the depth).
In a class with bounded expansion the requirement is stronger:
For every $r \geq 1$, depth-$r$ minors should be sparse, that is, their average degrees should be bounded by some function $f(r)$.

It is well-known that graphs with no $H$-minors have average degree bounded by a function of $H$.
Thus the corresponding class of graphs has bounded expansion, since the average degree of their minors is uniformly bounded by a constant.
The class of graphs with no topological $H$-minors also has bounded expansion,
though in that case the bounding function might not be constant anymore, see~\cite{NOdM-book}.
These remarks show that Theorem~\ref{thm:bounded-expansion} generalizes the aforementioned result for cover graphs excluding a fixed graph as a topological minor, and hence the previous body of work as well.

In Figure~\ref{fig:hierarchy} we give a summary of all the known results about posets with sparse cover graphs having their dimension bounded by a function of their height.
We also mention the cases where dimension is bounded by an absolute constant, see the bottom part of the figure.
Complementing our main result, we describe in Section~\ref{sec:prelim} a family of height-$2$ posets whose cover graphs form a nowhere dense class and having unbounded dimension.
Thus Theorem~\ref{thm:bounded-expansion} cannot be extended to nowhere dense classes.
The graphs constructed moreover have locally bounded treewidth, showing that this property alone is not sufficient either.
Finally, in order to complete our study of sparse cover graphs following the theory of Ne\v{s}et\v{r}il and Ossona de Mendez, we remark that another  specialization of nowhere dense classes called  \emph{almost wide}  classes also fail, in the sense that dimension is not bounded in terms of height for posets with cover graphs belonging to such a class. (We postpone the rather technical definition of almost wide classes until Section~\ref{sec:prelim}.)

\subsection{Applications}
Let us now turn our attention to some applications of Theorem~\ref{thm:bounded-expansion} in the context of graph drawing, where natural classes with bounded expansion appear that do not fit in the previous setting of excluding a (topological) minor.
Consider for instance posets whose diagrams can be drawn with `few' edge crossings.
If we bound the total number of crossings, then the cover graphs have bounded genus, and hence earlier results apply.
On the other hand, if we only bound the number of crossings \emph{per edge} in the drawing,
say at most $k$ such crossings, then we come to the well-studied class of \emph{$k$-planar} graphs.
While $k$-planar graphs are sparse---Pach and T\'{o}th~\cite{PT97} proved that their average degree is at most $8\sqrt{k}$---
they do not exclude any graph as a topological minor (assuming $k \geq 1$),
since every graph has a $1$-planar subdivision: just start with any drawing of the graph in the plane and subdivide its edges around each crossing.
Note however that this could require some edges to be subdivided many times.
Indeed, one can observe that if an $(<r)$-subdivision of a graph $G$ is $k$-planar then $G$ is $kr$-planar.
Combining this fact with Pach and T\'{o}th's result, Ne\v{s}et\v{r}il, Ossona de Mendez, and Wood~\cite{NOdMW12} proved that the class of $k$-planar graphs has bounded expansion.
Therefore, by Theorem~\ref{thm:bounded-expansion}, whenever a poset $P$ has a $k$-planar cover graph, $\dim(P)$ is bounded by a function of $k$ and the height of $P$.

Another example is given by graphs with bounded book thickness.
A \emph{book embedding} of a graph is a collection of half-planes (\emph{pages}) all having the same line as their boundary (the \emph{spine}) such that all vertices of the graph lie on the line, every edge is contained in one of the half-planes, and edges on a same page do not cross.
The \emph{book thickness} (also known as \emph{stack number}) of a graph is the smallest number of pages in a book embedding.
Clearly, every graph on $n$ vertices with book thickness $k$ has at most $2kn$ edges.
Still, every graph has a subdivision with book thickness at most $3$~(see e.g.~\cite{DW05}).
Ne\v{s}et\v{r}il {\it et al.}~\cite{NOdMW12} observed that a result of Enomoto, Miyauchi, and Ota~\cite{EMO99} easily implies that every class of graphs with bounded book thickness has bounded expansion.
Therefore, by Theorem~\ref{thm:bounded-expansion}, whenever a poset $P$ has a cover graph of book thickness at most $k$,  its dimension is bounded by a function of $k$ and the height of $P$.

Yet another class of graphs with bounded expansion from graph drawing is that of graphs with bounded queue number.
A \emph{queue layout} of a graph is an ordering of its vertices (the spine) together with an edge coloring such that there are no two nested monochromatic edges, where two edges are \emph{nested} if all four endpoints are distinct and the endpoints of one edge induce an interval on the spine containing the endpoints of the other edge.
Then the \emph{queue number} of a graph is the minimum number of colors in a queue layout.
Every graph has a subdivision with queue number $2$, as proved by Dujmovi\'c and Wood~\cite{DW05}, who also showed that graphs with bounded queue number form a class with bounded expansion.
A challenging open problem in the area of queue layouts is to decide whether planar graphs have constant queue number (see e.g.~\cite{Dujmovic-jctb-2015}).

\subsection{Proof overview}
In the proof we use a characterization of classes with bounded expansion in terms of $p$-centered colorings.
A \emph{$p$-centered coloring} of a graph $G$ is a vertex coloring of $G$ such that, for every connected subgraph $H$ of $G$, either some color is used exactly once in $H$, or at least $p$ colors are used in $H$.
Ne\v{s}et\v{r}il and Ossona de Mendez~\cite{NOdM-decomp} proved that a class $\calC$ has bounded expansion if and only if there is a function $f:\N \to \N$ such that for every integer $p \geq 1$ and every graph $G\in \calC$, there is a $p$-centered coloring of $G$ using at most $f(p)$ colors.
Given a poset $P$ of height $h$ with cover graph in $\calC$, we work with a $2h$-centered coloring $\phi$ of the cover graph with at most $f(2h)$ colors.

Given this coloring $\phi$, we focus on upsets of points $x\in P$: 
For each point $y$ such that $x \leq y$ in $P$, we consider a sequence of cover relations $x=z_1 < \cdots < z_k =y$ witnessing that $x \leq y$, which we call a {\em covering chain} (thus this is a path from $x$ going up to $y$ in the diagram of $P$).   
Such a covering chain defines a corresponding color sequence $\sigma=(\phi(z_1),\ldots,\phi(z_k))$, whose length is bounded by the height of $P$. 
We then exploit the fact that there are a bounded number of different color sequences and consider upsets of points  w.r.t.\ a fixed color sequence $\sigma$.  
Focusing on such upsets, we uncover a wealth of structure, culminating in a proof that a certain family of subsets of points of $P$ is {\em laminar}. 
In fact, we define one such laminar family for each subset of the set of color sequences. 
Once this is set up, we then exploit the underlying tree structure of these laminar families to bound the dimension of $P$.  

We note that the idea for the first part of the proof---working with colored-upsets and proving that certain well-chosen families of sets are laminar---comes directly from a recent paper of Reidl, S\'anchez Villaamil, and Stavropoulos~\cite{ReidlVS16}, who used it to obtain a new characterization of classes with bounded expansion in terms of `neighborhood complexity'. 
Indeed, their proof method turned out to be perfectly fitted to approach our problem in a clean and simple way. 
As a result, the proof we present in this paper is shorter and simpler than the one we gave in the preliminary version of this paper~\cite{SODAversion}. 
Without going into details, that proof relied on a powerful decomposition of posets into layers called \emph{unfolding}, introduced in the paper of Streib and Trotter~\cite{ST14}, that emerged as a key technical tool in recent papers on poset dimension. 
We see it as an appealing feature of our proof that it avoids unfolding entirely. 
In fact, no background nor tool from poset theory is needed, except for the elementary notion of an alternating cycle (see Section~\ref{sec:posets}).  \\

The paper is organized as follows.
In Section~\ref{sec:prelim} we give the necessary definitions for all sparse classes of graphs depicted in Figure~\ref{fig:hierarchy}, as well as the necessary definitions regarding posets.
Along the way, we also describe constructions of posets showing that
Theorem~\ref{thm:bounded-expansion} cannot be extended to the classes in grey in Figure~\ref{fig:hierarchy}. 
In Section~\ref{sec:main_proof} we give the proof of Theorem~\ref{thm:bounded-expansion}.

\section{Definitions and preliminaries}
\label{sec:prelim}

\subsection{Posets}\label{sec:posets}
Let us start with basic notions about posets.
All posets considered in this paper are finite.
Elements of a poset $P$ are called \emph{points}.
Points $x,y\in P$ are said to be \emph{comparable} in $P$ if $x\leq y$ or $x\geq y$ in $P$.
Otherwise $x$ and $y$ are \emph{incomparable} in $P$. 
A set of points $C\subseteq P$ is a \emph{chain} in $P$ if the points in $C$ are pairwise comparable.
The \emph{height} of $P$ is the maximum size of a chain in $P$.
We write $x<y$ in $P$ if it holds that $x\leq y$ and $x\neq y$.
For distinct $x,y\in P$, if $x<y$ in $P$ and there is no $z\in P$ with $x<z<y$ in $P$ then $x<y$ is a \emph{cover relation} of $P$.
By $\cover(P)$ we denote the \emph{cover graph} of $P$, the (undirected) graph defined on the points of $P$ where edges correspond to cover relations of $P$.

A \emph{linear extension} $L$ of $P$ is a linear order on the ground set of $P$ such that $x\leq y$ in $L$ whenever $x \leq y$ in $P$.
Linear extensions $L_1,\ldots,L_d$ form a \emph{realizer} of $P$ if their intersection is equal to $P$, that is, $x\leq y$ in $P$ if and only if $x\leq y$ in $L_i$ for each $i\in\set{1,\ldots,d}$.
The \emph{dimension} of $P$, denoted by $\dim(P)$, is the least number $d$ such that there is a realizer of $P$ of size $d$.

We let $\Inc(P)=\{(x,y) \in P\times P\mid x \textrm{ is incomparable to $y$ in } P\}$ denote the set of ordered pairs of incomparable points in $P$. 
A set $I \subseteq \Inc(P)$ of incomparable pairs is \emph{reversible} if there is a linear extension $L$ of $P$ that reverses each pair in $I$, that is, $y<x$ in $L$ for every $(x,y)\in I$.

We can rephrase the definition of dimension as follows: Assuming $\Inc(P) \neq \emptyset$, the dimension of $P$ is the least positive integer $d$ for which there exists a partition of $\Inc(P)$ into $d$ reversible sets.  
(Note that $\dim(P)=1$ in case $\Inc(P)=\emptyset$.)   
In light of this statement, it is handy to have a simple criterion to decide when a set of incomparable pairs is reversible. 
This motivates the following definition. 
A sequence $(x_1,y_1), \dots, (x_k,y_k)$ of  $k \geq 2$ pairs from $\Inc(P)$  such that $x_i\leq y_{i+1}$ holds in $P$ for all $i\in\set{1,\ldots,k}$ (cyclically) is called an \emph{alternating cycle}. 
It is well known (and easy to show) that a set $I$ of incomparable pairs is reversible if and only if $I$ does not contain any alternating cycle (see e.g.\ \cite{Tro-book}).  
Indeed, this is the only fact about dimension that we will need in our proof. 

\subsection{Sparse graph classes}
Next, we introduce the necessary definitions regarding graphs and give proper definitions for
the graph classes mentioned in the introduction.
All graphs in this paper are finite, simple, and undirected.
Given a graph $G$ we denote by $V(G)$ and $E(G)$ the vertex set and edge set of $G$, respectively. 
$H$ is a {\em subgraph} of $G$ if $V(H) \subseteq V(G)$ and $E(H) \subseteq V(G)$. 
For a subset $X\subset V(G)$ we denote by $G[X]$ the subgraph of $G$ with vertex $X$ and all edges with both endpoints in $X$.    
The \emph{distance} between two vertices in $G$ is the length of a shortest path between them.
(Thus adjacent vertices are at distance $1$; also, distance between two vertices in distinct components of $G$ is set to $+\infty$.)
The set of all vertices at distance at most $r$ from vertex $v$ in $G$ is denoted by $N^r_G(v)$, and the subscript is omitted if $G$ is clear from the context.
The \emph{radius} of a connected graph $G$ is the least integer $r \geq 0$ for which there is a vertex $v\in V(G)$ such that $N^r_G(v)=V(G)$.

The \emph{treewidth} of a graph $G$, denoted by $\tw(G)$,  is the least integer $t \geq 0$ for which there is a tree $T$ and a family $\set{T(v)\mid v\in V(G)}$ of non-empty subtrees of $T$ such that
$\norm{\set{v\in V(G) \mid x\in T(v)}}\leq t+1$ for each node $x$ of $T$, and
$V(T(u)) \cap V(T(v)) \neq\emptyset$ for each edge $uv\in E(G)$.
A class of graphs $\calC$ has \emph{locally bounded treewidth} if there exists a function $f:\R\to \R$ such that
$\tw(G[N^r(v)])\leq f(r)$ for every integer $r\geq 0$, graph $G\in\calC$ and vertex $v\in V(G)$.

Given a partition $\calX$ of the vertices of a graph $G$ into non-empty parts inducing connected subgraphs,
we denote by $G/\calX$ the graph with vertex set $\calX$ and edge set defined as follows:
For two distinct distinct parts $X, Y \in \calX$, we have $XY\in E(G/\calX)$ if and only if there exist $x\in X$ and $y\in Y$ such that $xy\in E(G)$.
A graph $H$ is a \emph{minor} of $G$ if $H$ is isomorphic to a subgraph of $G/\calX$ for some such partition $\calX$ of $V(G)$.
A specialization of this notion is that of topological minors: $H$ is a \emph{topological minor} of a graph $G$
if $G$ contains a subgraph isomorphic to a subdivision of $H$. 
(A subdivision of $H$ is any graph that can be obtained from $H$ by replacing each $uv$ with a path $P_{uv}$ having $u$ and $v$ as endpoints and whose internal vertices are new vertices of the graph, that is, the paths $P_{uv}$ ($uv \in E(H)$) are internally vertex-disjoint.)    
A class of graphs $\calC$ is \emph{minor closed} (\emph{topologically closed}) if every minor (topological minor, respectively) of a graph in $\calC$ is also in $\calC$.

We pursue with the definitions of classes with bounded expansion and nowhere dense classes.
A graph $H$ is a \emph{depth-$r$ minor} (also known as an \emph{$r$-shallow minor})  of a graph $G$
if $H$ is isomorphic to a subgraph of $G/\calX$ for some partition $\calX$ of $V(G)$ into
non-empty parts inducing subgraphs of radius at most $r$.
The \emph{greatest reduced average density} (\emph{grad}) of \emph{rank $r$} of a graph $G$, denoted by $\nabla_r(G)$, is defined as
$\nabla_r(G)=\max\left\{\frac{\norm{E(H)}}{\norm{V(H)}} \mid \textrm{$H$ is a depth-$r$ minor of $G$}\right\}$.
A class of graphs $\calC$ has \emph{bounded expansion}  if there exists a function $f:\R\to \R$ such that
$\nabla_r(G)\leq f(r)$ for every integer $r\geq 0$ and graph $G\in\calC$.
A class of graphs $\calC$ is \emph{nowhere dense} if for each integer $r\geq 0$ there exists a graph
which is {\em not} a depth-$r$ minor of any graph $G\in\calC$.

It is easy to see that classes with locally bounded treewidth and classes with bounded expansion are nowhere dense.
Note however that these two notions are incomparable.
Two classical examples of classes with locally bounded treewidth are graphs with bounded maximum degree, and graphs excluding some apex graph $A$ as a minor.
The latter is in fact a characterization of {\em minor-closed} classes with locally bounded treewidth:
A minor-closed class $\calC$ has locally bounded treewidth if and only if $\calC$ excludes some apex graph,
a fact originally proved by Eppstein~\cite{Eppstein-lbtw}.

Let us now define almost wide classes.
For $d\geq1$, a set of vertices $X$ in a graph $G$ is \emph{$d$-independent} if every two distinct vertices in $X$ are at distance strictly greater than $d$ in $G$.
A class of graphs $\calC$ is \emph{almost wide}
if there exists an integer $s \geq 0$ such that
for every integer $d\geq 1$ there is a function $f:\R\to \R$ such that
for every integer $m\geq 1$, every graph $G\in\calC$ of order at least $f(m)$ contains a subset $S$ of size at most $s$ so that $G-S$ has a $d$-independent set of size  $m$.
Ne\v{s}et\v{r}il and Ossona de Mendez~\cite[Theorem~3.23]{NOdM10} proved that a class of graphs excluding a fixed graph as a topological minor is almost wide.
Moreover, they proved~\cite[Theorem~3.13]{NOdM10} that every \emph{hereditary} class of graphs that is almost wide is also nowhere dense.
(Recall that a class is hereditary if it is closed under taking induced subgraphs.)

\subsection{Posets with cover graphs in a nowhere dense class}
As mentioned in the introduction, the statement of Theorem~\ref{thm:bounded-expansion} cannot be pushed further towards nowhere dense classes.
In fact, it already fails for classes with locally bounded treewidth and hereditary classes that are almost wide, as we now show.
Our construction is based on the class of graphs $G$ with $\Delta(G) \leq \girth(G)$, where
$\Delta(G)$ and $\girth(G)$ denote the maximum degree and girth of $G$, respectively.
This is a useful example of a hereditary class with locally bounded treewidth that is also almost wide but that does not have bounded expansion.
(Indeed, this class is invoked several times in the textbook~\cite{NOdM-book}.)
We will use in particular that the chromatic number of these graphs is unbounded.
This is a well-known fact that can be shown in multiple ways;
we can note for instance that the chromatic number of the $n$-vertex $d$-regular non-bipartite Ramanujan graphs with girth $\Omega_{d}(\log n)$ built by Lubotzky,  Phillips, and Sarnak~\cite{LPS88} have chromatic number
$\Omega(\sqrt{d})$ (see~\cite{LPS88}).

\begin{proposition}\label{prop:locally-bounded-tw}
There exists a hereditary almost wide class of graphs $\calC$ with locally bounded treewidth
such that posets of height $2$ with cover graphs in $\calC$ have unbounded dimension.
\end{proposition}
\begin{proof}
For a graph $G$, the \emph{adjacency poset} $P_G$ of $G$ is the poset with point set $\set{a_v \mid v\in V(G)}\cup\set{b_v \mid v\in V(G)}$ such that, for every two distinct vertices $u, v \in V(G)$, we have
$a_u \leq b_v$ in $P_G$ if and only if $uv\in E(G)$.
It is well known that $\dim(P_G) \geq \chi(G)$, see~\cite{FT00}.  
This can be seen as follows.
Fix a realizer $L_1,\ldots,L_d$ of $P_G$.
For every vertex $v\in V(G)$, fix a number $\phi(v)=i$ such that $b_v < a_v$ in $L_i$.
We claim that $\phi$ is a proper coloring of $G$.
Consider any two adjacent vertices $u$ and $v$ in $G$ and, say, $\phi(v)=i$.
Then $a_u \leq b_v < a_v \leq b_u$ in $L_i$, which witnesses that $\phi(u)\neq i$.
Therefore, $\dim(P_G) \geq \chi(G)$.

Now let $\calC$ denote the class of graphs $G$ satisfying  $\Delta(G) \leq \girth(G)$.
As mentioned earlier, this class is hereditary, almost wide, and has locally bounded treewidth.
This is not difficult to check (or see~\cite{NOdM-book} for a proof).

The key observation about the class $\calC$ in this context
is that if $G \in \calC$, then  $\cover(P_G) \in \calC$.
This can be seen as follows: First, clearly $\Delta(G) = \Delta(\cover(P_G))$, so it is enough to show
$\girth(G) \leq \girth(\cover(P_G))$.
To show the latter, we remark that if $C$ is a cycle of $\cover(P_G)$, then
$C$ naturally corresponds to a closed walk $W$ in $G$ of the same length.
Moreover, every three consecutive vertices in that walk $W$ are pairwise distinct, as follows from the adjacency poset construction.
Hence, $W$ contains a cycle, which is of length at most that of $C$.
Therefore, $\girth(G) \leq \girth(\cover(P_G))$, as claimed.

To summarize, graphs in $\calC$ have unbounded chromatic number, implying that adjacency posets of
these graphs have unbounded dimension.
Yet, the cover graphs of these adjacency posets all belong to $\calC$, a  hereditary almost wide class with locally bounded treewidth.
This concludes the proof.
\end{proof}

\section{Proof of main theorem}
\label{sec:main_proof}

There are a number of equivalent conditions for classes of graphs to have bounded expansion (see \cite{NOdMW12} for instance).
As mentioned in the introduction, we use the following characterization in terms of $p$-centered colorings: 
A class $\calC$ has bounded expansion if and only if there exists a function $f:\N \to \N$ such that for every integer $p \geq 1$ and every graph $G\in \calC$ there is a $p$-centered coloring of $G$ using at most $f(p)$ colors.
Thus, the following theorem implies Theorem~\ref{thm:bounded-expansion}.

\begin{theorem}\label{thm:p-centered} 
If $P$ is a poset of height $h$ and its cover graph has a $2h$-centered coloring using $c$ colors, then
\[
\dim(P) \leq 2^{2(c+1)^h}. 
\]
\end{theorem}

This section is devoted to the proof of Theorem~\ref{thm:p-centered}.  
Let thus $P$ be a poset of height $h$, and let $\phi$ be a $2h$-centered coloring of the cover graph of $P$ using colors from the set $\set{1,\ldots,c}$. 

\subsection{Signatures for covering chains}
First, we refine the coloring $\phi$ of the cover graph of $P$ as follows.
For each $x \in P$, let $h(x)$ denote the \emph{height} of $x$ in $P$, that is the size of a longest chain in $P$ ending with $x$.
For each $x\in P$ we define
\[
 \phi^*(x):=(\phi(x),h(x)).
\]
Note that $\phi^*$ is a $2h$-centered coloring as in general any refinement of a $p$-centered coloring is still a $p$-centered coloring, for every $p\geq1$.

We say that $Q=(x_1,x_2,\ldots,x_{\ell})$ is a \emph{covering chain} of $P$ if $x_i<x_{i+1}$ is a cover relation in $P$ for each $i\in\set{1,\ldots,\ell-1}$.
In this case, we also say that $Q$ is a covering chain \emph{from} $x_1$ \emph{to} $x_\ell$.
Each covering chain $Q=(x_1,x_2,\ldots,x_\ell)$ has a \emph{signature} defined by
\[
 (\phi^*(x_1),\ldots,\phi^*(x_\ell)).
\]
When $\sigma=(\phi^*(x_1),\ldots,\phi^*(x_\ell))$ is the signature of $Q$, we also call $Q$ a \emph{$\sigma$-covering chain}.
By our definition of $\phi^*$ we have that the signature of $Q$ is always  \emph{proper}, in the sense that no color appears more than once in the signature.

Let $\Sigma$ be the set of all signatures of covering chains of $P$.
We can bound the size of $\Sigma$ as follows: $|\Sigma| \leq \sum {h \choose k} c^k = (c+1)^h$. 
For each $\sigma\in\Sigma$, we let $X_{\sigma}$ denote the set of all points $x$ in $P$ such that there is a $\sigma$-covering chain starting from $x$.
For illustration, each $x$ in $P$ belongs to $X_{\sigma}$ for $\sigma=((\phi(x),h(x)))$.
For $x$ in $P$, the \emph{$\sigma$-upset} $U_\sigma(x)$ of $x$ is the set of all points $y$ in $P$ such that there is a $\sigma$-covering chain from $x$ to $y$.
Similarly, the \emph{$\sigma$-downset} $D_{\sigma}(y)$ of $y$ is the set of all points $x$ in $P$ such that there is a $\sigma$-covering chain from $x$ to $y$. 

\subsection{Two lemmas on \texorpdfstring{$\sigma$}--upsets}

\begin{lemma}\label{lem:neighborhoods}
	Let $\sigma \in \Sigma$ and let $x,x'\in X_{\sigma}$. 
 If a $\sigma$-covering chain starting in $x$ intersects a $\sigma$-covering chain starting in $x'$, then $U_\sigma(x)=U_\sigma(x')$.
 In particular, if $U_\sigma(x)\cap U_\sigma(x')\neq \emptyset$, then $U_\sigma(x)=U_\sigma(x')$.
\end{lemma}
\begin{proof}
 Let $Q_{\sigma}^{x,y}$, $Q_{\sigma}^{x',y'}$ be two $\sigma$-covering chains from $x$ to $y$ and from $x'$ to $y'$, respectively.
 Suppose that these two paths share a common point $z$.
 We need to prove that $U_\sigma(x)=U_\sigma(x')$.
 Suppose for contradiction that $U_\sigma(x)\neq U_\sigma(x')$,  and assume without loss of generality that $U_\sigma(x')- U_\sigma(x)\neq \emptyset$.
 Let $y''\in U_\sigma(x') - U_\sigma(x)$. 
 Let $Q_{\sigma}^{x',y''}$ be a $\sigma$-covering chain from $x'$ to $y''$ (see Figure~\ref{fig:neighborhoods} for an illustration).

\begin{figure}[t]
\centering
\includegraphics{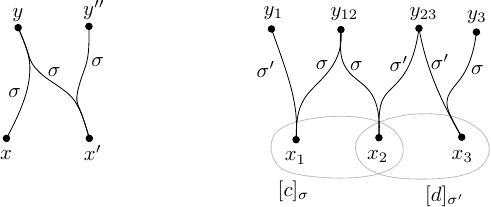}
\caption{Illustration for possible situations in Lemma~\ref{lem:neighborhoods} (left) and Lemma~\ref{lem:laminar} (right).}
\label{fig:neighborhoods}
\end{figure}

 Clearly, the union $U$ of the three paths $Q_{\sigma}^{x,y}$, $Q_{\sigma}^{x',y}$ and $Q_{\sigma}^{x',y''}$ is a connected subgraph of the cover graph of $P$.
 Since these three paths have the same signature, we know that $U$ is colored by $\phi^*$ with at most $h$ colors.
 
 We claim that the two paths $Q_{\sigma}^{x,y}$ and $Q_{\sigma}^{x',y''}$ are vertex disjoint.
 Suppose that $Q_{\sigma}^{x,y}$ and $Q_{\sigma}^{x',y''}$ share a common element $w$. 
 Combining the $x$--$w$ portion of $Q_{\sigma}^{x,y}$ 
 with the $w$--$y''$ portion of $Q_{\sigma}^{x',y''}$ we obtain a covering chain from $x$ to $y''$. 
Using that the signature $\sigma$ is proper, it is easy to see that this covering chain from $x$ to $y''$ also has signature $\sigma$.
 It follows that $y''\in U_\sigma(x)$, contradicting our choice of $y''$.
 Hence, $Q_\sigma^{x,y}$ and $Q_\sigma^{x',y''}$ are vertex disjoint, as claimed. 
 It follows that no color of $\phi^*$ is used exactly once on $U$, contradicting the fact that $\phi^*$ is a $2h$-centered coloring.  
\end{proof}

For $\sigma \in \Sigma$ and points $x,y\in P$, write $x\sim_\sigma y$ if $U_\sigma(x)$ and $U_\sigma(y)$ intersect.
By Lemma~\ref{lem:neighborhoods} we know that $\sim_\sigma$ defines an equivalence relation on $P$;
denote by $[x]_{\sigma}$ the equivalence class of point $x$ with respect to $\sim_\sigma$.

For $\hat\Sigma \subseteq \Sigma$, let  $X_{\hat\Sigma}$ denote the set of points $x$ of $P$ for which $\hat\Sigma$ is exactly the subset of signatures $\sigma \in \Sigma$ 
such that there is a $\sigma$-covering chain starting at $x$, that is, 
\[
X_{\hat\Sigma} := \set{x\in P\mid  \set{\sigma \in \Sigma \mid U_\sigma(x)\neq \emptyset} = \hat\Sigma}.
\]

Let $\calF_{\hat\Sigma}$ be the family of subsets of $X_{\hat\Sigma}$ defined as the union of the equivalence classes of $\sim_\sigma$ for
$\sigma \in\hat \Sigma$ restricted to the set $X_{\hat\Sigma}$. 
That is,
\[
\calF_{\hat\Sigma} := \set{[x]_{\sigma}\cap X_{\hat\Sigma} \mid x\in X_{\hat\Sigma},\ \sigma\in\hat\Sigma}. 
\]
A key observation is that these families are laminar.
Recall that a family $\calF$ is \emph{laminar} if for every $C,D\in\calF$ we have $C\cap D =\emptyset$, or $C\subset D$, or $D\subset C$.

\begin{lemma}\label{lem:laminar}
The family $\calF_{\hat\Sigma}$ is laminar for every $\hat\Sigma\subset\Sigma$.
\end{lemma}

\begin{proof}
Let $\hat\Sigma \subset\Sigma$ and let $C,D\in\calF_{\hat\Sigma}$.
We need to prove that $C\cap D =\emptyset$, or $C\subset D$, or $D\subset C$. 
Let $c,d\in X_{\hat\Sigma}$ and $\sigma,\sigma' \in \hat\Sigma$ be such that $C=[c]_{\sigma}\cap X_{\hat\Sigma}$ and $D=[d]_{\sigma'}\cap X_{\hat\Sigma}$. 
In order to get a contradiction, suppose that  $C-D\neq\emptyset$, $C\cap D \neq\emptyset$, and $D-C\neq\emptyset$. 
Thus, there are elements $x_1,x_2,x_3\in X_{\hat\Sigma}$ such that
$x_1 \in [c]_{\sigma} - [d]_{\sigma'}$, $x_2 \in [c]_{\sigma} \cap [d]_{\sigma'}$, and $x_3 \in [d]_{\sigma'}- [c]_{\sigma}$.
It follows that $[x_1]_{\sigma} = [c]_{\sigma} = [x_2]_{\sigma}$ and $[x_2]_{\sigma'}= [d]_{\sigma'} = [x_3]_{\sigma'}$.
Since $x_1,x_2,x_3 \in X_{\hat\Sigma}$ and $\sigma,\sigma'\in \hat\Sigma$ we know that the $\sigma$-upsets and $\sigma'$-upsets of $x_1$, $x_2$ and $x_3$ are not empty.
Hence, there exist $y_{12} \in U_{\sigma}(x_1)=U_{\sigma}(x_2)$, $y_{23} \in U_{\sigma'}(x_2)=U_{\sigma'}(x_3)$, $y_1 \in U_{\sigma'}(x_1)$, and $y_3 \in U_{\sigma}(x_3)$.
Let $Q_{\sigma'}^{x_1,y_1}$ denote a $\sigma'$-covering chain from $x_1$ to $y_1$,  and define $Q_{\sigma}^{x_1,y_{12}}$, $Q_{\sigma}^{x_2,y_{12}}$, $Q_{\sigma'}^{x_2,y_{23}}$, $Q_{\sigma'}^{x_3,y_{23}}$, $Q_{\sigma}^{x_3,y_{3}}$ similarly.

Consider the union of all these paths: 
\[
U:=Q_{\sigma'}^{x_1,y_1}\cup Q_\sigma^{x_1,y_{12}}\cup Q_{\sigma}^{x_2,y_{12}}\cup Q_{\sigma'}^{x_2,y_{23}}\cup Q_{\sigma'}^{x_3,y_{23}}\cup Q_{\sigma}^{x_3,y_3}.
\]
(See Figure~\ref{fig:neighborhoods} for an illustration). 
Clearly, $U$ is a connected subgraph of the cover graph of $P$, and all colors of $\phi^*$ used on $U$ come from $\sigma$ and $\sigma'$.
Both $\sigma$ and $\sigma'$ consist of at most $h$ colors and we know moreover that they share at least one color, namely $\phi^*(x_1)=\phi^*(x_2)=\phi^*(x_3)$.
This means that $U$ is colored by $\phi^*$ with at most $2h-1$ colors.

Next, we show that each color of $\sigma$ appears at least twice on $U$.
To do so, we show that the two paths $Q_\sigma^{x_2,y_{12}}$ and $Q_\sigma^{x_3,y_3}$ are vertex disjoint.
Indeed, if these two paths intersect then by Lemma~\ref{lem:neighborhoods} we would have $U_{\sigma}(x_2) = U_{\sigma}(x_3)$ and therefore $[x_3]_{\sigma} = [x_2]_{\sigma} = [c]_{\sigma}$, contradicting the choice of $x_3$.

Analogously, we argue that each color of $\sigma'$ appears at least twice on $U$ by showing that the two paths $Q_{\sigma'}^{x_1,y_1}$ and $Q_{\sigma'}^{x_2,y_{23}}$ are vertex disjoint.
Indeed, if these two paths intersect then by Lemma~\ref{lem:neighborhoods} we would have $U_{\sigma'}(x_1) = U_{\sigma'}(x_2)$ and therefore $[x_1]_{\sigma'} = [x_2]_{\sigma'} = [d]_{\sigma'}$, contradicting the choice of $x_1$.

This shows that among the at most $2h-1$ colors used by $\phi^*$ on $U$, none appears exactly once, contradicting the fact that $\phi^*$ is a $2h$-centered coloring.
\end{proof}

Using the family $\calF_{\hat\Sigma}$, we define a linear order $\prec_{\hat\Sigma}$ on $X_{\hat\Sigma}$.  
This linear order results from the following recursive procedure applied to $Z=X_{\hat\Sigma}$:  
If no set in $\calF_{\hat\Sigma}$ is a proper subset of $Z$, then order the elements in $Z$ arbitrarily. 
Otherwise, let $S_1,\ldots,S_{\ell}$ denote the inclusion-wise maximal sets in $\calF_{\hat\Sigma}$ that are proper subsets of $Z$.  
Since $\calF_{\hat\Sigma}$ is laminar no two of the sets $S_1,\ldots,S_{\ell}$ intersect. 
Moreover, the union of $S_1,\ldots,S_{\ell}$ is $Z$. 
(Indeed, $S_1 = [x]_{\sigma} \cap X_{\hat\Sigma}$ for some $x\in X_{\hat\Sigma}$ and $\sigma \in \hat\Sigma$. 
Thus, for each $y\in Z-S_1$, the set $Y:= [y]_{\sigma} \cap X_{\hat\Sigma}$ intersects $Z$ and avoids $S_1$, which implies $Y \in \calF_{\hat\Sigma}$ and $y \in Y \subsetneq Z$, by laminarity.)  
We set $x \prec_{\hat\Sigma} y$ for all $1\leq i<j\leq \ell$ and all $x\in S_i$ and $y\in S_j$.
The relative ordering of the elements within $S_i$ ($i\in \{1,\dots, \ell\}$) is then obtained by applying the recursive procedure to $S_i$.  

For distinct elements $x,y\in X_{\hat\Sigma}$ we say that $x$ \emph{lies to the left} of $y$ in $X_{\hat\Sigma}$ if $x\prec_{\hat\Sigma} y$, and \emph{to the right} of $y$ in $X_{\hat\Sigma}$ otherwise. 
A crucial property of $\prec_{\hat\Sigma}$ is that for every $x\in X_{\hat\Sigma}$ and $\sigma \in \hat\Sigma$ the points of $[x]_{\sigma}\cap X_{\hat\Sigma}$ form an interval in $\prec_{\hat\Sigma}$.
We emphasize this property with the following proposition.

\begin{proposition}\label{prop:interval}
	Let $\hat\Sigma \subset\Sigma$, $\sigma\in\hat\Sigma$ and $x_1,x_2, x_3\in X_{\hat\Sigma}$.
	If $x_1\prec_{\hat\Sigma} x_2\prec_{\hat\Sigma} x_3$ and $[x_1]_{\sigma}\cap X_{\hat\Sigma} = [x_3]_{\sigma}\cap X_{\hat\Sigma}$, 
	then $[x_1]_{\sigma}\cap X_{\hat\Sigma}=[x_2]_{\sigma}\cap X_{\hat\Sigma}=[x_3]_{\sigma}\cap X_{\hat\Sigma}$.
\end{proposition}

\subsection{Partitioning the incomparable pairs}
Using the tools developed so far, we now show how to partition the set of incomparable pairs of $P$ into a number of reversible sets that depends only on $h$.  
For $\hat\Sigma \subseteq \Sigma$ let 
$\Inc_{\hat\Sigma}(P) := \set{(x,y)\in\Inc(P) \mid x\in X_{\hat\Sigma},\ y\in P}$.
Clearly,
\[
\Inc(P)=\bigcup_{\hat\Sigma\subseteq\Sigma}\Inc_{\hat\Sigma}(P).
\]
Since $|\Sigma|\leq (c+1)^h$, there are at most $2^{(c+1)^h}$ sets in the above partition of $\Inc(P)$. 
Note that this partition of $\Inc(P)$ is based solely on properties of the first element of pairs in $\Inc(P)$. 
We are going to refine this partition by considering properties of the second element of pairs in $\Inc(P)$. 
We need the following observation: 

\begin{lemma}\label{lem:downset-left-or-right}
Let $\hat\Sigma \subseteq \Sigma$, $\sigma\in \hat\Sigma$, and $(x,y) \in \Inc_{\hat\Sigma}(P)$. 
Then no two points $x_1,x_2 \in D_{\sigma}(y) \cap X_{\hat\Sigma}$ are such that $x_1 \prec_{\hat\Sigma} x \prec_{\hat\Sigma} x_2$.
\end{lemma}
\begin{proof}
Suppose to the contrary that there are points $x_1,x_2 \in D_{\sigma}(y) \cap X_{\hat\Sigma}$ such that $x_1 \prec_{\hat\Sigma} x \prec_{\hat\Sigma} x_2$.
Since $y\in U_{\sigma}(x_1) \cap U_{\sigma}(x_2)$, we have $[x_1]_{\sigma} = [x_2]_{\sigma}$.
But by~Proposition~\ref{prop:interval} this yields $[x_1]_{\sigma} = [x]_{\sigma} = [x_2]_{\sigma}$, forcing $x \leq y$ in $P$, a contradiction.
\end{proof}

For each $\hat\Sigma \subset \Sigma$ and $(x,y)\in \Inc_{\hat\Sigma}(P)$ we let $v(x,y)$ be the binary vector $(v_\sigma(x,y))_{\sigma\in\hat\Sigma}$ where

\begin{align*}
v_\sigma(x,y)=\begin{cases}
          0 & \text{ if no point of }D_\sigma(y)\cap X_{\hat\Sigma}\text{ lies to the right of }x \\
          1 & \text{ otherwise}.
         \end{cases}
\end{align*}

Let $$V_{\hat\Sigma} := \{ v(x,y) \mid (x,y) \in \Inc_{\hat\Sigma}(P) \}.$$ 
For $v\in V_{\hat\Sigma}$, let $\Inc_{\hat\Sigma,v}(P) = \set{(x,y)\in \Inc(P) \mid x\in X_{\hat\Sigma}\ \text{and}\ v(x,y)=v}$.
Clearly,
\[
\Inc(P)=\bigcup_{\hat\Sigma\subseteq\Sigma,v\in V_{\hat\Sigma}}\Inc_{\hat\Sigma,v}(P).
\]
We show that each set in this partition of $\Inc(P)$ is reversible.
\begin{lemma}
$\Inc_{\hat\Sigma,v}(P)$ is reversible for every $\hat\Sigma\subset\Sigma$ and $v\in V_{\hat\Sigma}$.
\end{lemma}
\begin{proof}
Let $\hat\Sigma\subset\Sigma$ and $v\in V_{\hat\Sigma}$, and   
suppose for a contradiction that $\Inc_{\hat\Sigma,v}(P)$ is not reversible. 
Then $\Inc_{\hat\Sigma,v}(P)$ contains an alternating cycle $(x_1,y_1),\dots,(x_k,y_k)$. 
Shifting cyclically the pairs if necessary, we may assume that $x_2$ is leftmost among all $x_i$'s with respect to $\prec_{\hat\Sigma}$. 
Consider a covering chain from $x_1$ to $y_2$ and another from $x_2$ to $y_3$ (if $k=2$ then $y_3=y_1$). 
Say they have signatures $\sigma$ and $\sigma'$ respectively, and denote them by $Q_\sigma^{x_1,y_2}$ and $Q_{\sigma'}^{x_2,y_3}$.

Clearly, $\sigma,\sigma'\in \hat\Sigma$.
Since $x_2\prec_{\hat\Sigma} x_1$ and $x_1\in D_{\sigma}(y_2) \cap  X_{\hat\Sigma}$, we conclude that $v_\sigma(x_2,y_2)=1$.
Similarly since $x_2\prec_{\hat\Sigma} x_3$ and  $x_2 \in D_{\sigma'}(y_3) \cap  X_{\hat\Sigma}$, we have $v_{\sigma'}(x_3,y_3)=0$ by Lemma~\ref{lem:downset-left-or-right}. 
As $v(x_2,y_2)=v=v(x_3,y_3)$ we also obtain that $v_\sigma(x_3,y_3)=1$ and $v_{\sigma'}(x_2,y_2)=0$.
(In particular, this shows $\sigma\neq \sigma'$.)

Since $v_\sigma(x_3,y_3)=1$ we know that there is $x'\in X_{\hat\Sigma}$ with $x_3\prec_{\hat\Sigma} x'$ such that there is a $\sigma$-covering chain $Q_\sigma^{x',y_3}$ from $x'$ to $y_3$.
Since $x'\in X_{\hat\Sigma}$ and $\sigma'\in \hat\Sigma$, there is also a $\sigma'$-covering chain $Q_{\sigma'}^{x',y'}$ from $x'$ to $y'$, for some $y'$ in $P$.
And finally, since $x_2 \in X_{\hat\Sigma}$ and $\sigma\in\hat\Sigma$, there is a $\sigma$-covering chain $Q_\sigma^{x_2,y''}$ from $x_2$ to $y''$, for some $y''$ in $P$.
See Figure~\ref{fig:alt-cycle} for a possible situation.
Now we consider the union
\[
 U:= Q_\sigma^{x_2,y''}\cup Q_{\sigma'}^{x_2,y_3}\cup Q_\sigma^{x',y_3}\cup Q_{\sigma'}^{x',y'}.
\]
$U$ is a connected subgraph of the cover graph of $P$, and its vertives are colored by $\phi^*$ with colors coming from $\sigma$ and $\sigma'$.
Since $\sigma$ and $\sigma'$ share at least one common color, namely $\phi^*(x_2) = (\phi(x_2), h(x_2))$, we see that at most $2h-1$ different colors appear on $U$.

We claim that the two paths $Q_\sigma^{x_2,y''}$ and $Q_\sigma^{x',y_3}$ are vertex disjoint.
Suppose not. 
Then by Lemma~\ref{lem:neighborhoods}  $[x_2]_{\sigma} = [x']_{\sigma}$.
Since $x_2\prec_{\hat\Sigma} x_3\prec_{\hat\Sigma} x'$, by Proposition~\ref{prop:interval} this implies that $[x']_{\sigma} \cap X_{\hat\Sigma}=[x_3]_{\sigma} \cap X_{\hat\Sigma}$.
However, this implies in turn that $U_\sigma(x') = U_\sigma(x_3)$, and in particular $y_3\in U_\sigma(x_3)$, contradicting the fact that $x_3$ and $y_3$ are incomparable in $P$. 

Next we claim that $Q_{\sigma'}^{x_2,y_3}$ and $Q_{\sigma'}^{x',y'}$ are also vertex disjoint. 
Arguing again by contradiction, suppose it is not the case. 
Then $[x_2]_{\sigma'} = [x']_{\sigma'}$ by Lemma~\ref{lem:neighborhoods}.
Since $x_2\prec_{\hat\Sigma} x_3\prec_{\hat\Sigma} x'$, by Proposition~\ref{prop:interval} it follows that $[x_2]_{\sigma'} \cap X_{\hat\Sigma}=[x_3]_{\sigma'} \cap X_{\hat\Sigma}$, which implies  $y_3\in U_{\sigma'}(x_2) = U_{\sigma'}(x_3)$, contradicting the fact that $x_3$ and $y_3$ are incomparable.

We conclude that every color from $\sigma$ is used at least twice on $U$ (since  $Q_\sigma^{x_2,y''}$ and $Q_\sigma^{x',y_3}$ are disjoint), and that the same holds for colors from $\sigma'$ (since $Q_{\sigma'}^{x_2,y_3}$ and $Q_{\sigma'}^{x',y'}$ are disjoint). 
This contradicts the fact that $\phi^*$ is a $2h$-centered coloring of the cover graph of $P$. 
\end{proof}

\begin{figure}[t]
\centering
\includegraphics{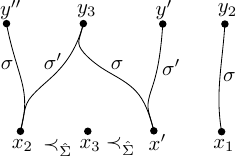}
\caption{Illustration of the covering chains induced by an alternating cycle.}
\label{fig:alt-cycle}
\end{figure}

Since $\norm{\Sigma} \leq (c+1)^h$ we have at most $2^{(c+1)^h}$ ways to choose $\hat\Sigma \subset \Sigma$. 
Also, for $\hat\Sigma\subset\Sigma$  we have $|V_{\hat\Sigma}| \leq  2^{|\hat\Sigma|} \leq 2^{|\Sigma|}\leq 2^{(c+1)^h}$. 
Therefore, we can bound the dimension of $P$ as follows:  
\[
	\dim(P)\leq 2^{(c+1)^h} \cdot 2^{(c+1)^h} = 2^{2(c+1)^h}.
\]

This concludes the proof of Theorem~\ref{thm:p-centered}. 

\section{Concluding Remarks}
Our main theorem reveals a connection between the notions of bounded expansion and order dimension.
We think that this might actually provide a new characterization of classes with bounded expansion: 
\begin{conjecture}
Let $\mathcal{C}$ be a class of graphs closed under taking subgraphs. 
Then $\mathcal{C}$ has bounded expansion if and only if for each $h\geq 1$, posets of height $h$ whose cover graphs are in $\mathcal{C}$ have bounded dimension.  
\end{conjecture}
Note that Theorem~\ref{thm:bounded-expansion} shows one of the two directions of this conjecture, the other one remains open.

As mentioned earlier, classes with bounded expansion can be characterized in various ways by asking that certain coloring numbers are bounded (this includes e.g.\ {\em low tree-depth colorings}, {\em generalised colorings}, and {\em centered colorings}, see~\cite{NOdMW12}). 
In most cases, if one asks instead the coloring numbers to be subpolynomially bounded in terms of the number $n$ of vertices of the graph then this results in a characterization of nowhere dense classes. 
For instance,  a class $\calC$ is nowhere dense if and only if for every $\varepsilon>0$ and $p\geq1$, every $n$-vertex graph in $\calC$ has a $p$-centered coloring with $O(n^{\varepsilon})$ colors; see~\cite{GKS13} for a survey of such characterizations.  
Perhaps one could similarly establish a subpolynomial bound on the dimension of bounded-height posets with cover graphs in a nowhere dense class?  
(This was suggested by Dan Kr\'{a}l.)
 
\begin{problem}\label{prob:nowhere-dense}
Is it true that if $\mathcal{C}$ is a nowhere dense class of graphs, then for every $\epsilon>0$ and $h\geq 1$, posets on $n$ points of height $h$ and with cover graphs in $\mathcal{C}$ have dimension $O(n^{\epsilon})$?
\end{problem}

\section*{Acknowledgements}

We thank David R.\ Wood for suggesting that Theorem~\ref{thm:bounded-expansion} might be true, which prompted us to work on this question, and Bartosz Walczak for his insightful remarks on the proof.  
We are also grateful to the referees of the current version of this paper and of its preliminary version~\cite{SODAversion} for their numerous and helpful comments.

\nocite{TM77}
\nocite{FTW13}
\nocite{biro}
\nocite{JMTWW}

\bibliographystyle{plain}
\bibliography{posets-dimension}

\begin{thebibliography}{10}

\bibitem{biro}
Csaba Bir\'o, Mitchel~T. Keller, and Stephen~J. Young.
\newblock Posets with cover graph of pathwidth two have bounded dimension.
\newblock {\em Order}, 33(2):195--212, 2016.
\newblock \href{http://arxiv.org/abs/1308.4877}{arXiv:1308.4877}.

\bibitem{CLN12}
Parinya Chalermsook, Bundit Laekhanukit, and Danupon Nanongkai.
\newblock Graph products revisited: tight approximation hardness of induced
  matching, poset dimension and more.
\newblock In {\em Proceedings of the {T}wenty-{F}ourth {A}nnual {ACM}-{SIAM}
  {S}ymposium on {D}iscrete {A}lgorithms}, pages 1557--1576. SIAM,
  Philadelphia, PA, 2012.

\bibitem{Dujmovic-jctb-2015}
Vida Dujmovi{\'c}.
\newblock Graph layouts via layered separators.
\newblock {\em J. Combin. Theory Ser. B}, 110:79--89, 2015.
\newblock \href{http://arxiv.org/abs/1302.0304}{arXiv:1302.0304}.

\bibitem{DW05}
Vida Dujmovi{\'c} and David~R. Wood.
\newblock Stacks, queues and tracks: layouts of graph subdivisions.
\newblock {\em Discrete Math. Theor. Comput. Sci.}, 7(1):155--201, 2005.

\bibitem{DM41}
Ben Dushnik and Edwin~W. Miller.
\newblock Partially ordered sets.
\newblock {\em Amer. J. Math.}, 63:600--610, 1941.

\bibitem{EMO99}
Hikoe Enomoto, Miki~Shimabara Miyauchi, and Katsuhiro Ota.
\newblock Lower bounds for the number of edge-crossings over the spine in a
  topological book embedding of a graph.
\newblock {\em Discrete Appl. Math.}, 92(2-3):149--155, 1999.

\bibitem{Eppstein-lbtw}
David Eppstein.
\newblock Diameter and treewidth in minor-closed graph families.
\newblock {\em Algorithmica}, 27(3-4):275--291, 2000.
\newblock \href{http://arxiv.org/abs/math/9907126}{arXiv:math/9907126}.

\bibitem{FT00}
Stefan Felsner and William~T. Trotter.
\newblock Dimension, graph and hypergraph coloring.
\newblock {\em Order}, 17(2):167--177, 2000.

\bibitem{FTW13}
Stefan Felsner, William~T. Trotter, and Veit Wiechert.
\newblock The {D}imension of {P}osets with {P}lanar {C}over {G}raphs.
\newblock {\em Graphs Combin.}, 31(4):927--939, 2015.

\bibitem{FK86}
Zoltan F{\"u}redi and Jeff Kahn.
\newblock On the dimensions of ordered sets of bounded degree.
\newblock {\em Order}, 3(1):15--20, 1986.

\bibitem{GKS13}
Martin Grohe, Stephan Kreutzer, and Sebastian Siebertz.
\newblock Characterisations of nowhere dense graphs.
\newblock In {\em 33nd {I}nternational {C}onference on {F}oundations of
  {S}oftware {T}echnology and {T}heoretical {C}omputer {S}cience}, volume~24 of
  {\em LIPIcs. Leibniz Int. Proc. Inform.}, pages 21--40. Schloss Dagstuhl.
  Leibniz-Zent. Inform., Wadern, 2013.

\bibitem{JMMTWW}
Gwena\"{e}l Joret, Piotr Micek, Kevin~G. Milans, William~T. Trotter, Bartosz
  Walczak, and Ruidong Wang.
\newblock Tree-width and dimension.
\newblock {\em Combinatorica}, 36(4):431--450, 2016.
\newblock \href{http://arxiv.org/abs/1301.5271}{arXiv:1301.5271}.

\bibitem{JMTWW}
Gwena\"{e}l Joret, Piotr Micek, William~T. Trotter, Ruidong Wang, and Veit
  Wiechert.
\newblock On the dimension of posets with cover graphs of treewidth $2$.
\newblock {\em Order}, to appear.
\newblock \href{http://arxiv.org/abs/1406.3397}{arXiv:1406.3397}.

\bibitem{SODAversion}
Gwena\"{e}l Joret, Piotr Micek, and Veit Wiechert.
\newblock Sparsity and dimension.
\newblock In {\em Proceedings of the Twenty-Seventh Annual ACM-SIAM Symposium
  on Discrete Algorithms}, SODA '16, pages 1804--1813. SIAM, 2016.
\newblock \href{http://arxiv.org/abs/1507.01120v1}{arXiv:1507.01120v1}.

\bibitem{Kel81}
David Kelly.
\newblock On the dimension of partially ordered sets.
\newblock {\em Discrete Math.}, 35:135--156, 1981.

\bibitem{LPS88}
Alexander Lubotzky, Ralph Phillips, and Peter Sarnak.
\newblock Ramanujan graphs.
\newblock {\em Combinatorica}, 8(3):261--277, 1988.

\bibitem{MW}
Piotr Micek and Veit Wiechert.
\newblock Topological minors of cover graphs and dimension.
\newblock Submitted, \href{http://arxiv.org/abs/1504.07388}{arXiv:1504.07388}.

\bibitem{NOdM-decomp}
Jaroslav Ne{\v{s}}et{\v{r}}il and Patrice Ossona~de Mendez.
\newblock Grad and classes with bounded expansion. {I}. {D}ecompositions.
\newblock {\em European J. Combin.}, 29(3):760--776, 2008.
\newblock \href{http://arxiv.org/abs/math/0508323}{arXiv:math/0508323}.

\bibitem{NOdM10}
Jaroslav Ne{\v{s}}et{\v{r}}il and Patrice Ossona~de Mendez.
\newblock First order properties on nowhere dense structures.
\newblock {\em J. Symbolic Logic}, 75(3):868--887, 2010.

\bibitem{NOdM-book}
Jaroslav Ne{\v{s}}et{\v{r}}il and Patrice Ossona~de Mendez.
\newblock {\em Sparsity}, volume~28 of {\em Algorithms and Combinatorics}.
\newblock Springer, Heidelberg, 2012.
\newblock Graphs, structures, and algorithms.

\bibitem{NOdMW12}
Jaroslav Ne{\v{s}}et{\v{r}}il, Patrice Ossona~de Mendez, and David~R. Wood.
\newblock Characterisations and examples of graph classes with bounded
  expansion.
\newblock {\em European J. Combin.}, 33(3):350--373, 2012.
\newblock \href{http://arxiv.org/abs/0902.3265}{arXiv:0902.3265}.

\bibitem{PT97}
J{\'a}nos Pach and G{\'e}za T{\'o}th.
\newblock Graphs drawn with few crossings per edge.
\newblock {\em Combinatorica}, 17(3):427--439, 1997.

\bibitem{ReidlVS16}
Felix Reidl, Fernando~Sanchez Villaamil, and Konstantinos Stavropoulos.
\newblock Characterising bounded expansion by neighbourhood complexity.
\newblock \href{http://arxiv.org/abs/1603.09532}{arXiv:1603.09532}.

\bibitem{ST14}
Noah Streib and William~T. Trotter.
\newblock Dimension and height for posets with planar cover graphs.
\newblock {\em European J. Combin.}, 35:474--489, 2014.

\bibitem{Tro-book}
William~T. Trotter.
\newblock {\em Combinatorics and partially ordered sets}.
\newblock Johns Hopkins Series in the Mathematical Sciences. Johns Hopkins
  University Press, Baltimore, MD, 1992.
\newblock Dimension theory.

\bibitem{Tro-handbook}
William~T. Trotter.
\newblock Partially ordered sets.
\newblock In {\em Handbook of combinatorics, {V}ol.\ 1,\ 2}, pages 433--480.
  Elsevier Sci. B. V., Amsterdam, 1995.

\bibitem{TM77}
William~T. Trotter, Jr. and John~I. Moore, Jr.
\newblock The dimension of planar posets.
\newblock {\em J. Combinatorial Theory Ser. B}, 22(1):54--67, 1977.

\bibitem{Walczak17}
Bartosz Walczak.
\newblock Minors and dimension.
\newblock {\em J. Combin. Theory Ser. B}, 122:668--689, 2017.
\newblock \href{http://arxiv.org/abs/1407.4066}{arXiv:1407.4066}.

\bibitem{Yan82}
Mihalis Yannakakis.
\newblock The complexity of the partial order dimension problem.
\newblock {\em SIAM J. Algebraic Discrete Methods}, 3(3):351--358, 1982.

\end{thebibliography}

\end{document}